\documentclass[11pt]{amsart}

\usepackage{amsmath,amssymb,amsthm}
\usepackage{graphicx}
\usepackage{subfigure}
\usepackage{enumerate}
\usepackage{fullpage}
\usepackage{url}
\usepackage{float}
\usepackage{xspace}
\usepackage{hyperref}
\usepackage{color}
\usepackage{bbm}

\newtheorem{theorem}{Theorem}
\newtheorem{conjecture}{Conjecture}
\newtheorem{corollary}{Corollary}
\newtheorem{lemma}{Lemma}[section]
\newtheorem{proposition}{Proposition}

\title{Covering by homothets and illuminating convex bodies}

\author{Alexey Glazyrin}
\address{School of Mathematical and Statistical Sciences, The University of Texas Rio Grande Valley, Brownsville, TX 78520}
\email{alexey.glazyrin@utrgv.edu}

\date{}%\today}

\begin{document}

\maketitle

\begin{abstract}
The paper is devoted to coverings by translative homothets and illuminations of convex bodies. For a given positive number $\alpha$ and a convex body $B$, $g_{\alpha}(B)$ is the infimum of $\alpha$-powers of finitely many homothety coefficients less than 1 such that there is a covering of $B$ by translative homothets with these coefficients. $h_{\alpha}(B)$ is the minimal number of directions such that the boundary of $B$ can be illuminated by this number of directions except for a subset whose Hausdorff dimension is less than $\alpha$. In this paper, we prove that $g_{\alpha}(B)\leq h_{\alpha}(B)$, find upper and lower bounds for both numbers, and discuss several general conjectures. In particular, we show that $h_{\alpha} (B) > 2^{d-\alpha}$ for almost all $\alpha$ and $d$ when $B$ is the $d$-dimensional cube, thus disproving the conjecture from \cite{bra06}.
\end{abstract}

\section{Introduction}\label{sect:intro}

Let $B\subset \mathbb{R}^d$ be a convex compact body with a non-empty interior. The family of homothets $\mathcal{F}=\{\lambda_1 B, \lambda_2 B,\ldots\}$, $\lambda_i \in (0,1)$, forms a translative covering of $B$ if $B\subseteq \cup_i (\lambda_i B + x_i)$, where $x_i$ are translation vectors in $\mathbb{R}^d$. The general question is to find necessary conditions on coefficients $\lambda_i$ for existence of a translative covering. We define $g_\alpha(B)$ and $g_\alpha(d)$ as follows.

$$g_\alpha(B)=\inf\left\{\sum\limits_{i=1}^k \lambda_i^\alpha: B\subseteq \bigcup\limits_{i=1}^k (\lambda_i B+x_i), \lambda_i\in(0,1), x_i\in\mathbb{R}^d\right\}$$

$$g_\alpha(d)=\inf\left\{g_\alpha(B): B\subset\mathbb{R}^d, B\text{ is a convex body}\right\}$$

Soltan in 1990 formulated the following conjecture which is also stated in \cite[Conjecture 2 of Section 3.2]{bra06}.

\begin{conjecture}[Soltan]\label{con:soltan}

$$g_1(d)\geq d.$$

\end{conjecture}

In \cite{sol93}, Conjecture \ref{con:soltan} was proven for the case $d=2$. In \cite{nas10}, the asymptotic version of the conjecture was proven for any $\alpha$, namely, it was shown that $\lim\limits_{d\rightarrow \infty} \frac{g_\alpha(d)}{d} = 1$ for a fixed $\alpha$. In \cite{gla18}, Conjecture \ref{con:soltan} was proven for the case of $d$-dimensional Euclidean balls, i.e. it was shown that $g_1(\mathbb{B}^d)=d$. The generalized version of the conjecture was also formulated in \cite{gla18} and shown to be tight for Euclidean balls in the case it is true.

\begin{conjecture}\label{conj:general}
For all natural $d$ and all $\alpha$ such that $0\leq \alpha \leq d+1$, $g_\alpha(d)=d+1-\left\lceil{\alpha}\right\rceil$.
\end{conjecture}

For each $d$-dimensional convex body $B$, $g_0(B)$, as defined above, stands for the minimal number of smaller homothets sufficient to cover $B$. There is extensive literature devoted to finding and bounding this number for various cases of convex bodies. The biggest open problem about this number is the Hadwiger conjecture (also known as the Levi-Hadwiger conjecture or the Gohberg-Marcus-Hadwiger conjecture, see \cite{lev55, had57, goh60} and surveys \cite[Section 3.3]{bra06}, \cite{mar99,bez15}).

\begin{conjecture}[Levi, Hadwiger, Gohberg-Marcus]\label{conj:hadw}
For any convex $d$-dimensional body $B$, $g_0(B)\leq 2^d$.
\end{conjecture}

It is easy to check that $g_0(B)=2^d$ for a $d$-dimensional parallelepiped and it is also conjectured that $g_0(B)< 2^d$ for all other convex bodies. Conjecture \ref{conj:hadw} is known to be true for centrally symmetric bodies in $\mathbb{R}^3$ \cite{las84, sol86} and several special cases of convex bodies in higher dimensions \cite{sch88, bol90, bol92, bol95, bez97a, tik17}.

The general upper bound $\binom{2d}{d} d \ln d (1+o(1))$ for $g_0(B)$ and the upper bound $2^d d \ln d (1+o(1))$ for centrally symmetric convex bodies are direct consequences of Rogers' lower bound on the density of coverings of the Euclidean space by translates of a convex body \cite{rog57} and the Rogers-Shephard inequality \cite{rog57a} (see \cite{bol97, rog97, bor04} for details and \cite{liv16} for a different approach leading to the same bound). For a long time, only very minor improvements of this bound were known \cite{fej09}. Recently, the sub-exponential asymptotic improvement of this bound was shown in \cite{hua18}: $g_0(B)\leq \binom{2d}{d} e^{-c\sqrt{d}}$ for some universal positive constant $c$.

Boltyanski \cite{bol60} and Hadwiger \cite{had60} found the connection between the problem of covering by smaller homothets and illumination problems. Here we follow the approach of Boltyanski. For a convex $d$-dimensional body $B$ we say that its boundary point $x$ is illuminated by an oriented direction $l$ if the ray from $x$ along the direction $l$ intersects the interior of $B$. The set of directions $\mathcal{L}=\{l_1,\ldots,l_k\}$ is said to illuminate $B$ if each boundary point of $B$ is illuminated by some direction from $\mathcal{L}$. By $h(B)$ we denote the minimal cardinality of $\mathcal{L}$ illuminating $B$. Then, as shown by Boltyanski, $h(B)=g_0(B)$.

We suggest an approach generalizing the illumination approach of Boltyanski. Denote by $h_\alpha(B)$ the minimal size of the set of directions $\mathcal{L}$ such that it illuminates all boundary points of $B$ except for a subset whose Hausdorff dimension is less than $\alpha$. In this paper we establish the connection between $g_{\alpha}$ and $h_{\alpha}$ and find new upper and lower bounds for these numbers in the spirit of classical results on covering (illumination) numbers $g_0$ ($h_0$).

The paper is organized as follows. In Section \ref{sect:illum} we show that $g_{\alpha}(B)\leq h_{\alpha}(B)$ and prove several statements about illumination numbers $h_{\alpha}$. We prove that $h_{\alpha} (B)\geq d+1-\left\lceil{\alpha}\right\rceil$ for a $d$-dimensional convex body $B$ and discuss the case of smooth $B$. Section \ref{sect:low_bounds} is devoted to the lower bounds on covering numbers $g_{\alpha}$. In particular, we prove that, for $d\geq 3$, $g_{\alpha}(B)\geq d-\left\lceil{\alpha}\right\rceil \ln^2 d$ and find lower bounds for direct products of convex bodies. In Section \ref{sect:cubes}, we explain why $d$-dimensional cubes provide a counterexample to the general version of the Hadwiger conjecture formulated for $\alpha>1$ in \cite[Section 3.3]{bra06}. In Section \ref{sect:up_bounds}, we use the new covering bound \cite{hua18} and the classical approach of Rogers to find various upper bounds for covering numbers $g_{\alpha}$. It appears that for sufficiently large $\alpha$ covering by smaller homothets leads to exponential improvements compared to the standard covering number $g_0$. On the other hand, as we show in this section, for some $\alpha$ it can be more efficient to use homothets of two unequal sizes.

\section{Illuminating convex bodies}\label{sect:illum}

In this section, we establish the connection between covering numbers $g_{\alpha}$ and illumination numbers $h_{\alpha}$ and discuss various bounds for $h_{\alpha}$ in the spirit of classical results and conjectures. 

\begin{theorem}\label{thm:h-alpha}
For any convex body $B$, $g_\alpha(B)\leq h_\alpha(B)$.
\end{theorem}

\begin{figure}[t]
\centering
\includegraphics[width=0.6\textwidth]{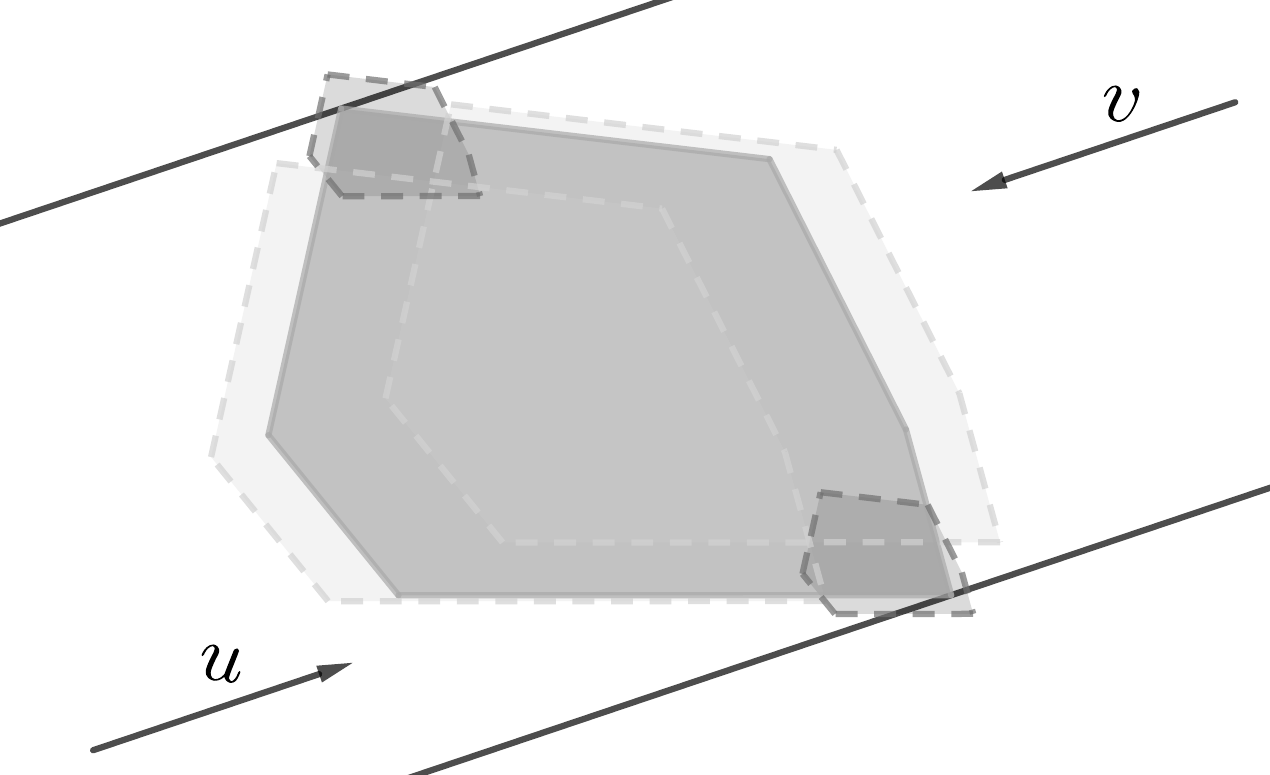}
\caption{Large homothets correspond to illumination vectors $u$ and $v$. The rest is covered by small homothets.}
\label{fig:cov}
\end{figure}

\begin{proof}
The general strategy of the proof is to cover most of the convex body $B$ by large homothets corresponding to illumination directions and use small homothets to cover the rest (see Figure \ref{fig:cov} for the 2-dimensional example).

Choose an arbitrary point in the interior of $B$ as the center $0$. Fix $\delta_1,\delta_2\in(0,1)$. We consider a set of directions $\mathcal{L}$ with precisely $h_{\alpha}(B)$ elements. For each unit vector $v$ defined by an illumination direction from $\mathcal{L}$, take an open homothet $(1-\delta_2)\,\text{int}\,(B)$ centered at $-\delta_1 v$. What remains uncovered in $B$ by the homothets for all illumination vectors is a closed set $B_{\delta_1,\delta_2}$. If a boundary point $y$ of $B$ belongs to a homothet $(1-\delta_2)\,\text{int}\,(B)-\delta_1v$, then $y+\delta_1 v \in (1-\delta_2)\,\text{int}\,(B)$. This implies $y+\delta_1 v \in \,\text{int}\,(B)$ and, subsequently, $y$ must be illuminated by the direction $v$. Hence $B_{\delta_1,\delta_2}$ contains all non-illuminated points on the boundary of $B$.

We want to show that $\bigcap\limits_{\delta_1,\delta_2\in(0,1)} B_{\delta_1,\delta_2}$, denoted by $B_0$, is precisely the set of all non-illuminated points on the boundary of $B$. Firstly, we can show that any interior point of $B$ does not belong to $B_0$. For such a point $x$, there is an open ball with the center $x$ and a radius $\epsilon$ contained in $B$. For simplicity, in this part of the proof, we consider the sets with $\delta_1=\delta_2=\delta$. Consider an arbitrary direction $v$ from $\mathcal{L}$. Then the homothet of $B$ with the center $-\delta v$ will contain the ball with the center $(1-\delta)x-\delta v$ and the radius $(1-\delta)\epsilon$.

$$\Vert (1-\delta)x-\delta v - x\Vert = \delta \Vert x+v \Vert\leq \delta \left(\sup_{x\in B} \Vert x \Vert +1\right),$$
which is less than $(1-\delta)\epsilon$ for a sufficiently small $\delta$ so the ball from the homothet contains $x$.

Now we want to show that any boundary point $y$, illuminated by a direction $v$, also does not belong to $B_0$. If $y$ is illuminated by $v$, then there is $\delta_1\in(0,1)$ such that $y+\delta_1v\in \,\text{int}\,(B)$. Then one can find $\delta_2\in(0,1)$ such that $\frac 1 {1-\delta_2 }(y+\delta_1v)\in \,\text{int}\,(B)$. Therefore, $y$ does not belong to $B_{\delta_1,\delta_2}$ and, subsequently, does not belong to $B_0$.

For the next part of the proof, we will show that, given $\delta_1,\delta_2,\Delta_1\in (0,1)$, $\Delta_1<\delta_1$, there exists $\Delta_2\in (0,1)$ such that the intersection of the homothet $-\delta_1 v + (1-\delta_2)\,\text{int}\,(B)$ with $B$ is a subset of $-\Delta_1 v + (1-\Delta_2)\,\text{int}\,(B)$. Consider $x$ from this intersection, i.e. $x\in B$ and $x=-\delta_1 v + (1-\delta_2)y$, where $y\in\,\text{int}\,(B)$. Then

$$x=\frac {\Delta_1} {\delta_1} x + \left(1-\frac {\Delta_1} {\delta_1}\right)x = \frac {\Delta_1} {\delta_1}\left( -\delta_1 v + (1-\delta_2)y\right) +\left(1-\frac {\Delta_1} {\delta_1}\right)x =$$

$$=-\Delta_1 v + \frac {\Delta_1} {\delta_1} (1-\delta_2)y + \left(1-\frac {\Delta_1} {\delta_1}\right)x$$
Taking $\Delta_2=\frac {\Delta_1 \delta_2} {\delta_1}$ and $\gamma=\frac {\Delta_1 (1-\delta_2)} {\delta_1-\Delta_1 \delta_2}$ so that $\Delta_2$ and $\gamma$ are both from $(0,1)$ and $\gamma y + (1-\gamma)x\in\,\text{int}\,(B)$, we get

$$x= -\Delta_1 v + (1-\Delta_2) (\gamma y + (1-\gamma)x) \in -\Delta_1 v + (1-\Delta_2)\,\text{int}\,(B).$$
The statement above implies that, given $\delta_1,\delta_2,\Delta_1\in (0,1)$, $\Delta_1<\delta_1$, there exists $\Delta_2\in (0,1)$ such that $B_{\Delta_1,\Delta_2}\subseteq B_{\delta_1,\delta_2}$.

Now consider an arbitrary open set $A$ containing $B_0$. Then $B\setminus A$ is a closed subset of $B\setminus B_0 = \bigcup\limits_{\delta_1,\delta_2\in(0,1)} B\setminus B_{\delta_1,\delta_2}$. Hence the sets $B\setminus B_{\delta_1,\delta_2}$ form an open covering of the compact set $B\setminus A$. Due to compactness, this open covering must contain a finite subset covering $B\setminus A$ too. Thus there exists a finite intersection $\bigcap\limits_{i=1}^N B_{\delta_{1i},\delta_{2i}}\subset A$. For $\Delta_1$ we choose an arbitrary positive number less than $\min\limits_{1\leq i\leq N} \{\delta_{1i}\}$. As $\Delta_2$ we take the minimal number that can be obtained by the method described above for all triples $\delta_{1i}, \delta_{2i}, \Delta_1$. This way we guarantee that $B_{\Delta_1,\Delta_2}\subseteq B_{\delta_{1i},\delta_{2i}}$ and, therefore $B_{\Delta_1,\Delta_2}\subset A$.

For the last step of the proof, we consider a covering of $B_0$ by open balls. Since the Hausdorff $\alpha$-content of $B_0$ is 0, for any $\varepsilon>0$ there exists such a covering with balls of radii $r_1,\ldots,r_M$ (the number is finite due to compactness of $B_0$) such that $\sum\limits_{j=1}^M r_j^{\alpha}<\varepsilon$. Assume an open ball of a radius $r$ with the center at $0$ is a subset of $B$. Then we substitute each ball of radius $r_j$ by an open homothet $\frac {r_j} {r}\,\text{int}\,(B)$. For the union of all these homothets, there exist, as was shown above, $\Delta_1$ and $\Delta_2$ such that $B_{\Delta_1,\Delta_2}$ is a subset of this union. If we add $h_{\alpha}(B)$ homothets covering $B\setminus B_{\Delta_1,\Delta_2}$, we obtain a covering of $B$. The sum of $\alpha$-powers of the homothety coefficients for this covering is

$$ h_{\alpha} (B) (1-\Delta_2)^\alpha + \sum\limits_{j=1}^M \left( \frac {r_j} {r}\right)^{\alpha}< h_{\alpha}(B) + \frac {\varepsilon} {r^{\alpha}}.$$

This value can be made as close to $h_{\alpha} (B)$ as desired and, therefore, $g_{\alpha}(B)\leq h_{\alpha}(B)$.
\end{proof}

Note that, for $\alpha=0$, $h_0(B)$ is just $h(B)$ and, due to Hadwiger and Boltyanski, there is equality in Theorem \ref{thm:h-alpha}. One can conjecture that this is the case for any natural $\alpha\leq d$ but we will show in Section \ref{sect:cubes} that this is not true even for cubes.

The results from \cite{bol60, goh60} also imply that $h(B)\geq d+1$ for any $d$-dimensional convex body $B$. We generalize these results for $h_\alpha(B)$.

\begin{theorem}\label{thm:genlowbound}
For any $d$-dimensional convex body $B$, $h_\alpha(B)\geq d+1-\left\lceil{\alpha}\right\rceil$.
\end{theorem}

For the proof of this theorem we will use the notion of \textit{shadow boundaries} (see \cite{mar96} for more details on the subject). For a $k$-dimensional linear subspace $L$ of $\mathbb{R}^d$, where $1\leq k\leq d-1$, by $\pi_L$ we mean an orthogonal projection along $L$ onto the subspace $L^{\perp}$. The shadow boundary $\Gamma (B,L)$ of a $d$-dimensional convex body $B$ along $L$ is the preimage of the relative boundary of $\pi_L (B)$ under $\pi_L^{-1}$. In terms of illumination, $\Gamma (B,L)$ is the set of points on the boundary of $B$ that are not illuminated by any direction from $L$.

\begin{proof}[Proof of Theorem \ref{thm:genlowbound}]
Since $h_\alpha(B)\geq h_{\left\lceil{\alpha}\right\rceil} (B)$, it is sufficient to prove the theorem for integer $\alpha$. We consider an arbitrary illumination of $B$ by $d-\alpha$ directions and we want to show that the set of non-illuminated boundary points of $B$ has a non-zero $\alpha$-dimensional Hausdorff content.

Let $L$ be a linear subspace generated by arbitrary $d-\alpha-1$ directions. Our goal is to show that the set of points not illuminated by directions from $L$ is at least $\alpha$-dimensional and the last direction cannot illuminate enough points to make the Hausdorff dimension smaller than $\alpha$. For the first part, we note that the dimension of $L$ is no greater than $d-\alpha-1$. The shadow boundary $\Gamma (B,L)$ is not illuminated by either of these directions. The $\alpha$-dimensional content of $\Gamma (B,L)$ is not 0 because the $\alpha$-dimensional content of the boundary of $\pi_L (B)$ is not 0.

To finish the proof it is sufficient to show that the Hausdorff dimension of the points from $\Gamma (B,L)$ not illuminated by the last direction is at least $\alpha$ as well. Denote the unit vector of this direction by $u$. If $v=\pi_L(u)$ is $0$, then $u$ does not illuminate any points from $\Gamma (B,L)$ so we assume $v$ is not 0. Denote by $V$ the line $\{tv\,\vert\,t\in\mathbb{R}\}$. We consider the projection $\pi_{L+V}=\pi_V\circ\pi_L$. The image of $B$ under $\pi_{L+V}$ has a non-zero $\alpha$-dimensional content. Now we note that for each point of this image its preimage under $\pi_V^{-1}$ contains a point from the boundary of $\pi_L(B)$ such that at least one exterior normal to the relative boundary of $\pi_L(B)$ has a non-negative scalar product with $v$. Taking a preimage of this point under $\pi_L^{-1}$ we find a point of $\Gamma (B,L)$ that is not illuminated by $u$. Hence for any point from $\pi_{L+V}(B)$, there is a point from $\Gamma (B,L)$ in its preimage under $\pi_{L+V}^{-1}$ which is not illuminated by $u$. Therefore, the set of all points not illuminated by $u$ has the dimension of at least $\alpha$ as required.
\end{proof}

For the next theorem, we will consider \textit{smooth} convex bodies. By a smooth $d$-dimensional convex body we mean a convex body such that there is a unique supporting hyperplane through each point of its boundary. The key property of illuminations of smooth convex bodies is the following lemma.

\begin{lemma}\label{lem:smooth}
The set of boundary points of a smooth convex body illuminated by directions defined by unit vectors $u_1,\ldots,u_M$ is the same as the set of boundary points of the body illuminated by directions defined by all unit vectors from the non-negaitive cone $\{a_1u_1+\ldots+a_Mu_M\,\vert\,a_i\geq 0\}$.
\end{lemma}

\begin{proof}
One set is trivially a subset of the other one. Hence we only need to show that if a point $x$ on the boundary of a smooth convex body $B$ is illuminated by the direction $a_1u_1+\ldots+a_ku_k$, $k\leq M$, all $a_i>0$, then it is illuminated by at least one of $u_1,\ldots,u_k$.

For a boundary point $x$ of $B$, there is a unique interior normal vector $n_x$. We can show that $x$ is illuminated by a unit vector $u$ if and only if $u\cdot n_x$ is positive. Indeed, $x$ is definitely not illuminated by the direction of $-u$ because $-u$ and $B$ are separated by the supporting hyperplane at $x$. If $x$ is not illuminated by $u$ as well, then the line defined by $u$ is a supporting line and is contained in some supporting hyperplane which is not orthogonal to $n_x$. This would contradict the uniqueness of the supporting hyperplane at $x$.

If $x$ is illuminated by the unit vector $u=a_1u_1+\ldots+a_ku_k$, all $a_i>0$, then $n_x\cdot (a_1u_1+\ldots+a_ku_k) >0$. Therefore, at least one of $n_x\cdot u_i$ is positive and $x$ is illuminated by $u_i$.
\end{proof}

The next result will be proven modulo a conjecture about shadow boundaries.

\begin{conjecture}\label{conj:shadow}
For a $d$-dimensional convex body $B$ and $k\leq d-1$, there exists a $k$-dimensional linear subspace $L$ of $\mathbb{R}^d$ such that the shadow boundary $\Gamma(B,L)$ has a finite $(d-k-1)$-dimensional Hausdorff measure.
\end{conjecture}

A more general question was formulated by McMullen who asked whether, given a $d$-dimensional convex body $B$ and $k\leq d-1$, for almost all $k$-dimensional linear subspaces $L$ shadow boundaries $\Gamma(B,L)$ have a finite $(d-k-1)$-dimensional Hausdorff measure. This question was answered positively by Steenaerts \cite{ste85} for the cases $k=1$ or $d-1$. To the best of the author's knowledge, the conjecture and its general version are open for other values of $k$.

Levi \cite{lev55} showed that the covering number $g_0(B)$ is precisely $d+1$ for a smooth $d$-dimensional convex body $B$ (this is also true for $B$ with a smooth belt \cite{dek94}). We will show that Conjecture \ref{conj:shadow} implies the general result for all positive $\alpha$.

\begin{theorem}\label{thm:smooth}
If Conjecture \ref{conj:shadow} is true, for any smooth $d$-dimensional convex body $B$, $h_\alpha(B) = d+1-\left\lceil{\alpha}\right\rceil$.
\end{theorem}

\begin{proof}
From Theorem \ref{thm:genlowbound}, we already know that $h_\alpha(B) \geq d+1-\left\lceil{\alpha}\right\rceil$ so, to complete the proof of the theorem, it is sufficient to find a set of $d+1-\left\lceil{\alpha}\right\rceil$ illuminating $B$.

Denote $M=d+1-\left\lceil{\alpha}\right\rceil$. We will consider sets of linearly dependent unit vectors $u_1,\ldots,u_M$ such that the non-negative cone formed by these vectors generates a linear subspace $L$ of dimension $M-1$. For each $L$ it is possible to find such a set of vectors. By Lemma \ref{lem:smooth}, the set of points not illuminated by $u_1,\ldots,u_M$ is precisely the shadow set $\Gamma(B,L)$.

If Conjecture \ref{conj:shadow} is true, there exists an $(M-1)$-dimensional linear subspace $L$ such that $\Gamma(B,L)$ has a finite $(d-M)$-dimensional Hausdorff content. This means that the dimension of $\Gamma(B,L)$ is not greater than $d-M=\left\lceil{\alpha}\right\rceil-1<\alpha$.
\end{proof}

Together with Theorem \ref{thm:smooth}, the results of Steenaerts mentioned above allow us to find $h_1(B)$ and $h_{d-1}(B)$ for all smooth bodies.

\begin{corollary}\label{cor:smooth}
For any smooth $d$-dimensional convex body $B$, $h_1(B) = d$ and $h_{d-1}(B)=2$.
\end{corollary}

We also note that $h_\alpha$ can be found for the case of Euclidean balls.

\begin{corollary}\label{cor:ball}
For all natural $d$ and all $\alpha$ such that $0\leq \alpha \leq d+1$, $h_\alpha(\mathbb{B}^d)=d+1-\left\lceil{\alpha}\right\rceil$ where $\mathbb{B}^d$ is the $d$-dimensional unit Euclidean ball.
\end{corollary}

\begin{proof}
Due to Theorem \ref{thm:smooth}, it is sufficient to show that the condition of Conjecture \ref{conj:shadow} is satisfied. A point $x$ on the unit sphere is illuminated by a direction $u$ if and only if $x\cdot u$ is negative. Assume $x$ is not illuminated by a $k$-dimensional linear subspace $L$ of $\mathbb{R}^d$. Then for any $u\in L$, $x\cdot u \geq 0$. On the other hand, $-u\in L$ so $x\cdot (-u)\geq 0$. We conclude that $x\cdot u=0$ for all $u\in L$ and any $x$ satisfying this condition belongs to the shadow boundary. Therefore, the shadow boundary $\Gamma(\mathbb{B}^d, L)$ is the intersection of the $(d-k)$ dimensional subspace $L^{\perp}$ with the unit sphere. This set has finite $(d-k-1)$-dimensional Hausdorff measure.
\end{proof}

Concerning the upper bounds for illumination numbers $h_{\alpha}$, the following conjecture was posed by Bezdek for $\alpha=1$ \cite{bez97} (see also \cite[Conjecture 3 of Section 3.3]{bra06}) and generalized in \cite[Section 3.3]{bra06} for all natural $\alpha$.

\begin{conjecture}\label{conj:hadw+}
For all non-negative integer $\alpha$ and all $d$-dimensional convex polytopes $B$, $h_\alpha(B)\leq 2^{d-\alpha}$.
\end{conjecture}

In Section \ref{sect:cubes}, we will show that this conjecture is in fact false for all $\alpha\in[2,d-2]$ for $d$-dimensional cubes when $d\geq 8$. We also note here that Conjecture \ref{conj:hadw+} is always true for $\alpha=d-1$ even for general convex bodies $B$. In order to prove this, it is sufficient to find two opposite directions, i.e. a linear subspace of dimension 1, such that the shadow boundary of $B$ with respect to this subspace has a zero $(d-1)$-dimensional Hausdorff measure. This result follows from the paper of Steenaerts \cite{ste85} mentioned above. The conjecture of Bezdek for $\alpha=1$ remains open.

\section{Lower bounds on covering numbers $g_\alpha$}\label{sect:low_bounds}

In this section we prove lower bounds on $g_\alpha(d)$. Although we do not know how to prove Conjecture \ref{conj:general}, we can still find out more about the sum of the powers of coefficients than the general asymptotic proved by Nasz{\'o}di \cite{nas10}.

\begin{theorem}\label{thm:const}
For $d\geq 3$ and any $\alpha\in[0,d]$, $g_\alpha(d) \geq d-\left\lceil{\alpha}\right\rceil \ln^2 d$.
\end{theorem}

\begin{proof}
It is sufficient to consider only integer $\alpha$. We will prove a slightly stronger statement. Instead of coverings of a convex body $B$, we will consider families of its homothets covering the boundary of $B$ and we will prove this lower bound for the sum of $\alpha$-powers of homothety coefficients for such families.

The proof is done by induction in $d$. The base is for $d=3$. For $\alpha=0$, at least 4 homothets are needed to cover the boundary of $B$. For $\alpha=1$, we can project everything to an arbitrary plane and,  from the proof of Soltan's conjecture for $d=2$ \cite{sol93}, use $g_1(2)\geq 2>3-\ln^2 3$. We know that the sum of surface areas of convex bodies is not smaller than the surface area of a convex body they cover. Using this with the inequality $1>3-2\ln^2 3$ we get the statement for $\alpha=2$. For $\alpha=3$, the right hand side of the inequality is negative. Altogether, the base of induction for $d=3$ is true.

For $\alpha$ such that $d-\alpha \ln^2 d\leq 1$, the statement holds immediately so we assume that $\alpha < \frac{d-1}{\ln^2 d}$. As mentioned earlier, the sum of surface areas of convex bodies is not smaller than the surface area of a convex body they cover. Hence homothety coefficients of the covering $\lambda_1$,$\ldots$, $\lambda_k$ satisfy $\sum\limits_{i=1}^k \lambda_i^{d-1} \geq 1$. Assume $\lambda_1$ is the largest of these coefficients. Then

$$\sum\limits_{i=1}^k \lambda_i^\alpha \geq \frac {\sum\limits_{i=1}^k \lambda_i^{d-1}} {\lambda_1^{d-\alpha-1}} \geq \frac {1} {\lambda_1^{d-\alpha-1}}.$$
If $\lambda_1^{d-\alpha-1} \leq \frac 1 {d-\alpha \ln^2 d}$ holds, the statement of the theorem is true. Hence it is sufficient to consider the case when $\lambda_1^{d-\alpha-1} > \frac 1 {d-\alpha \ln^2 d}\geq\frac 1 d$.

Let $l$ be the line connecting an arbitrary point in the interior of $B$ with its homothetic image in the largest homothet of the covering. We consider the orthogonal projection to $l^{\perp}$. Projections of homothets are homothets of the projection. Only those homothets which cover the boundary of the projection are being considered. By the induction hypothesis, the sum of the $\alpha$-powers of the homothety coefficients should be at least $(d-1)-\alpha\ln^2(d-1)$. Note that the projection of the largest homothet is strictly inside the projection of $B$ so it does not participate in covering the boundary of the projection of $B$. Hence we get

$$\sum\limits_{i=1}^k \lambda_i^\alpha \geq \lambda_1^\alpha + (d-1)-\alpha\ln^2(d-1) > \left(\frac 1 d \right)^{\frac {\alpha}{d-\alpha-1}}+(d-1)-\alpha\ln^2(d-1).$$

It remains to show that

$$\left(\frac 1 d \right)^{\frac {\alpha}{d-\alpha-1}}+(d-1)-\alpha\ln^2(d-1) \geq d-\alpha \ln^2 d;$$

$$\left(\frac 1 d \right)^{\frac {\alpha}{d-\alpha-1}}+\alpha(\ln^2 d-\ln^2(d-1))\geq 1.$$
For $d\geq 4$, $\ln^2 d-\ln^2(d-1) \geq \frac {2\ln d} d$. Hence it is sufficient to prove

$$\left(\frac 1 d \right)^{\frac {\alpha}{d-\alpha-1}} \geq 1- \frac {2\alpha\ln d} d.$$
Using $1- \frac {2\alpha\ln d} d \leq e^{- \frac {2\alpha\ln d} d } = \frac 1 {d^{2\alpha/d}}$, we are left with

$$\left(\frac 1 d \right)^{\frac 1 {d-\alpha-1}} \geq \left(\frac 1 d\right)^{\frac 2 d};$$

$$d-\alpha-1\geq \frac d 2,$$
which is true for any $d\geq 6$ and any real $\alpha< \frac{d-1}{\ln^2 d}$. For $d=4, 5$ we use that $\alpha$ must be integer and only $\alpha=0, 1$ satisfy the inequality $\alpha< \frac{d-1}{\ln^2 d}$ so the inequality above is true in these cases as well.
\end{proof}

We note that this result generalizes the result of Nasz{\'o}di \cite{nas10} since, for a fixed $\alpha$, $\lim\limits_{d\rightarrow\infty} \frac{d-\alpha \ln^2 d} d =1$, and also covers certain cases when $\alpha$ depends on $d$, for instance, $\alpha\sim d^c$ for all $c$ from $(0,1)$.

Due to Theorems \ref{thm:h-alpha} and \ref{thm:const} and Corollary \ref{cor:ball}, we know that $g_\alpha(d)$ is always squeezed between $d-\left\lceil{\alpha}\right\rceil\ln^2 d$ and $d+1-\left\lceil{\alpha}\right\rceil$. Generally, we conjecture (see Conjecture \ref{conj:general} in the introduction) that $g_\alpha(d)$ is precisely $d+1-\left\lceil{\alpha}\right\rceil$.

The following proposition, in some sense, confirms Conjecture \ref{conj:general} by showing that if the lower bound works for some convex body $P$ then it works for all direct products of $P$ and any other convex body for a certain range of values of $\alpha$.

\begin{proposition}\label{prop:direct}
For any $\alpha\in[0,d]$, a $d$-dimensional convex body $P$ and an $m$-dimensional convex body $Q$, $$g_{\alpha+m}(P\times Q)\geq g_\alpha (P).$$
\end{proposition}

\begin{proof}
Let $P\times Q$ be covered by $k$ smaller homothets $P_i\times Q_i$ with respective coefficients $\lambda_i$, $1\leq i\leq k$. Consider the indicator functions $\mathbbm{1}_i: Q\mapsto \{0,1\}$, $1\leq i\leq k$, such that $\mathbbm{1}_i(x)=1$ if $x\in Q_i$ and $\mathbbm{1}_i(x)=0$ otherwise. For a fixed $x\in Q$, the set of homothets such that $x\in Q_i$ forms a covering of $P\times \{x\}$. Therefore, $\sum\limits_{i=1}^k \mathbbm{1}_i (x) \lambda_i^\alpha \geq g_\alpha (P)$ for any $x\in Q$. Integrating this inequality over all points $x\in Q$ we get

$$\sum\limits_{i=1}^k \text{Vol}(Q_i) \lambda_i^\alpha \geq \int\limits_{x\in Q} \left(\sum\limits_{i=1}^k \mathbbm{1}_i (x) \lambda_i^\alpha\right) dx \geq \int\limits_{x\in Q} g_\alpha (P) dx = g_\alpha (P)\, \text{Vol}(Q).$$
Since $\text{Vol}(Q_i)=\lambda_i^m \text{Vol}(Q)$, we conclude that

$$\sum\limits_{i=1}^k \lambda_i^{\alpha+m} \geq g_{\alpha} (P).$$
\end{proof}

From this proposition, if Conjecture \ref{conj:general} is true for a $d$-dimensional convex body $P$ and $\alpha$, i.e. $g_\alpha (P)\geq d+1-\left\lceil{\alpha}\right\rceil$, then $g_{\alpha+m}(P\times Q) \geq d+1-\left\lceil{\alpha}\right\rceil = (d+m)+1-\left\lceil{\alpha+m}\right\rceil$ so Conjecture \ref{conj:general} is true for $P\times Q$ and $\alpha+m$.

\section{Generalized Hadwiger conjecture and cubes}\label{sect:cubes}

In this section we show that Conjecture \ref{conj:hadw+} does not generally hold for cubes.

\begin{theorem}\label{thm:cubes}
For any $d\geq 8$ and any positive integer $\alpha\in[2,d-2]$, $h_{\alpha} ([0,1]^d)>2^{d-\alpha}$.
\end{theorem}

\begin{proof}
Consider an arbitrary unit vector $a=(a_1,\ldots,a_d)$. The set of points on the surface of the cube illuminated by the direction of $a$ is completely defined by the signs of coordinates of $a$. Without loss of generality we assume $a_1,\ldots,a_l>0$, $a_{l+1},\ldots, a_m<0$, $a_{m+1}=\ldots=a_d=0$. Then the interior points of the $(d-m)$-dimensional face $T_a$ of the cube with coordinates $(t_1,\ldots,t_d)$ such that $t_1=\ldots=t_l=0$, $t_{l+1}=\ldots=t_m=1$, $t_{m+1},\ldots,t_d\in(0,1)$ are illuminated and all interior points of faces containing $T_a$ are illuminated as well. No other points are illuminated by the direction of $a$.

If, instead of the vector $a$, we use a unit vector $a'$ such that $a'_1,\ldots,a'_l>0$, $a'_{l+1},\ldots, a'_d<0$, the direction will illuminate at least as much on the surface of the cube as the direction of $a$ so we may assume that $a$ has no 0 coordinates. In this case $T_a$ is a vertex of the cube. Therefore, each illumination direction corresponds to the vertex defined by the signs of the vector of illumination. We consider the set of $h$ directions illuminating all interior points of $\alpha$-dimensional faces and denote their corresponding vertices by $x_1,\ldots,x_h$. We assume $h\leq 2^{d-\alpha}$ to get a contradiction at the end.

We would like all interior points of $\alpha$-dimensional faces to be illuminated so each $\alpha$-dimensional face contains at least one of $x_i$, $1\leq i\leq h$. At this point we note that the number of $\alpha$-faces of the $d$-dimensional cube is $\binom{d} {d-\alpha} 2^{d-\alpha}$ (the choice of $d-\alpha$ coordinates from $d$ options times the choice of 0 and 1 for each of these coordinates). The number of $\alpha$-dimensional faces containing a given vertex is $\binom{d} {d-\alpha}$ (the choice of $d-\alpha$ coordinates). Since $h\leq 2^{d-\alpha}$, we get that there is no $\alpha$-face containing at least two of the points $x_i$ and $h$ is precisely $2^{d-\alpha}$. This means that for any two points $x_i$ and $x_j$, $i\neq j$, $|x_i-x_j|\geq \alpha+1$ (in the Hamming metric) because otherwise they would both belong to the same $\alpha$-face defined by the $d-\alpha$ coordinates where they coincide.

The set of points $x_1,\ldots,x_h$ forms a binary code with $h=2^{d-\alpha}$ points and the minimal distance $\geq \alpha+1$. We can use any packing bounds for binary codes to check whether this is possible. One of the standard bounds is the Hamming bound claiming that the size of a binary $d$-dimensional code with the minimum Hamming distance $l$ is no greater than $\frac {2^d} {\sum\limits_{i=0}^t \binom{d}{i}}$, where $t=\left\lfloor\frac {l-1} 2\right\rfloor$ \cite{ham50}. Using this bound for our set of points we get

$$2^{d-\alpha}\leq \frac {2^d} {\sum\limits_{i=0}^{\left\lfloor\frac {\alpha} 2\right\rfloor} \binom{d}{i}};$$

$$\sum\limits_{i=0}^{\left\lfloor\frac {\alpha} 2\right\rfloor} \binom{d}{i} \leq 2^{\alpha}.$$

We want to get a contradiction and show that the inequality does not hold for $d\geq 8$ and $2\leq\alpha\leq d-2$. Since $\sum\limits_{i=0}^{\left\lfloor\frac {\alpha} 2\right\rfloor} \binom{d}{i}$ is an increasing function of $d$, it is sufficient to check that, for the five cases: 1) $\alpha=2$ and $d=8$, 2) $\alpha=3$ and $d=8$, 3) $\alpha=4$ and $d=8$, 4) $\alpha=5$ and $d=8$, and 5) $\alpha\geq 6$ and $d=\alpha+2$,

$$\sum\limits_{i=0}^{\left\lfloor\frac {\alpha} 2\right\rfloor} \binom{d}{i} > 2^{\alpha}.$$

In the first case, $\sum\limits_{i=0}^{1} \binom{8}{i} = 9 > 4 = 2^2$. In the second case, $\sum\limits_{i=0}^{1} \binom{8}{i} = 9 > 8 = 2^3$. In the third case, $\sum\limits_{i=0}^{2} \binom{8}{i} = 37 > 16 = 2^4$. In the fourth case, $\sum\limits_{i=0}^{2} \binom{8}{i} = 37 > 32 = 2^5$.

In the fifth case, we prove the inequality separately for even and odd $\alpha$. For the first part, assume $\alpha$ is even so $\alpha=2l$ and $d=2l+2$, $l\geq 3$. We want to prove $\sum\limits_{i=0}^{l} \binom{2l+2}{i} > 2^{2l}.$ Due to the symmetry of binomial coefficients, $\sum\limits_{i=0}^{l} \binom{2l+2}{i} = \frac 1 2 2^{2l+2} - \frac 1 2 \binom{2l+2}{l+1}$. Hence it is sufficient to show that $\binom{2l+2}{l+1}<2^{2l+1}$ for all $l\geq 3$. This inequality is true for $l=3$ and the inductive argument $\binom{2l+4}{l+2}< 4\binom{2l+2}{l+1}$ implies it is true for all larger $l$ as well.

In the second part, we assume $\alpha$ is odd so $\alpha=2l+1$, $d=2l+3$, $l\geq 3$ and we want to prove $\sum\limits_{i=0}^{l} \binom{2l+3}{i} > 2^{2l+1}.$ Here we use that $\sum\limits_{i=0}^{l} \binom{2l+3}{i} = \frac 1 2 2^{2l+3} - \binom{2l+3}{l+1}$ so it is sufficient to prove that $\binom{2l+3}{l+1}<2^{2l+1}$ for all $l\geq 3$. This inequality is true for $l=3$ and, similarly to the previous case, $\binom{2l+5}{l+2}< 4\binom{2l+3}{l+1}$ implies it is true for all larger $l$ too.

\end{proof}

Apart from giving a counterexample to Conjecture \ref{conj:hadw+}, Theorem \ref{thm:cubes} also shows that $h_{\alpha}$ is not always equal to $g_\alpha$. Due to Proposition \ref{prop:direct}, we know that for any integer $\alpha\in[0,d]$, $g_{\alpha} ([0,1]^d)\geq g_0 ([0,1]^{d-\alpha}) = 2^{d-\alpha}$. On the other hand, we can always cover a unit $d$-dimensional cube by $2^d$ twice smaller cubes so $g_{\alpha}([0,1]^d)\leq 2^d \left(\frac 1 2\right)^{\alpha} = 2^{d-\alpha}$. From these two observations, we conclude that $g_{\alpha}([0,1]^d)=2^{d-\alpha}$. By Theorem \ref{thm:cubes}, $h_{\alpha}([0,1]^d)>2^{d-\alpha}$ so $h_{\alpha}([0,1]^d)$ is not equal to $g_\alpha([0,1]^d)$ for almost all values of $\alpha$ and $d$.

What appeared to be false for illuminating numbers $h_{\alpha}$ still seems plausible for covering numbers $g_{\alpha}$ so we formulate the following generalized Hadwiger conjecture.

\begin{conjecture}\label{conj:hadw++}
For all $d$-dimensional convex bodies $B$ and all integer $\alpha\in[0,d]$, $g_\alpha(B)\leq 2^{d-\alpha}$.
\end{conjecture}

\section{Upper bounds on covering numbers $g_\alpha$}\label{sect:up_bounds}

In this section we prove upper bounds for $g_{\alpha}(B)$ using various approaches to the Hadwiger conjecture. As an easy application of the asymptotic bound for $g_0(B)$ \cite{hua18} and the trivial inequality $g_{\alpha}(B)\leq g_0(B)$, we immediately get the following proposition.

\begin{proposition}
There exist universal constants $c_1, c_2>0$ such that $$g_{\alpha}(B)\leq c_1 4^d e^{-c_2\sqrt{d}}$$ for any $d$-dimensional convex body $B$ and any positive number $\alpha$.
\end{proposition}

This trivial consequence of the covering bound may be improved when $\alpha$ is large enough with respect to $d$.

\begin{theorem}
There exist universal constants $c_1, c_2>0$ such that $$g_{\alpha}(B)\leq c_1 \frac {d^d} {\alpha^\alpha (d-\alpha)^{d-\alpha}}2^d e^{-c_2\sqrt{d}}$$ for any $d$-dimensional convex body $B$ and any positive number $\alpha>\frac d 2$.
\end{theorem}

\begin{proof}

In order to prove this theorem we use the statement following directly from \cite[Proof of Theorem 1.1, p. 9]{hua18}:
there is a universal constant $c>0$ such that for any $d$-dimensional convex body $B$, any $\lambda\in(0,1)$ and an arbitrary parameter $\beta\in(0,1)$, $B$ can be covered by no more than

$$\left(\frac {1+\beta\lambda} {\beta\lambda} \right)^d 2^d e^{-c\sqrt{d}} \left(1+d\ln\left( \frac 8 {(1-\beta)\lambda}\right) \right)$$
parallel translates of $\lambda B$. 

This statement immediately implies that
$$g_{\alpha}(B)\leq \lambda^{\alpha} \left(\frac {1+\beta\lambda} {\beta\lambda} \right)^d 2^d e^{-c\sqrt{d}} \left(1+d\ln\left( \frac 8 {(1-\beta)\lambda}\right) \right)$$
for any $\lambda, \beta\in (0,1)$.

Choosing $\beta=1-\frac 1 d$ and $\lambda=\frac {d-\alpha} {\alpha}$ we get the required bound.
\end{proof}

A similar improvement is possible in the case of centrally symmetric convex bodies. For the following theorem we use the smallest density $\theta(B)$ of a covering of the whole space $\mathbb{R}^d$ by translates of $B$.

\begin{theorem}\label{thm:sym1}
$$g_{\alpha}(B)\leq \frac {d^d} {\alpha^\alpha (d-\alpha)^{d-\alpha}}\theta(B)$$ for any $d$-dimensional convex centrally symmetric body $B$ and any positive number $\alpha>\frac d 2$.
\end{theorem}

\begin{proof}
We use one of the main results of \cite{rog97}: any convex $d$-dimensional body $B$ can be covered by no more than $\frac {\text{Vol} (B-H)} {\text{Vol}(H)}\theta(H)$ translates of a convex $d$-dimensional body $H$. Here we use this statement for a centrally symmetric $B$ and $H=\lambda B$ so $\theta(H)=\theta(B)$. Therefore, for any $\lambda\in(0,1)$,

$$g_{\alpha}(B)\leq \lambda^{\alpha} \frac {(1+\lambda)^d} {\lambda^d} \theta(B).$$

Taking $\lambda=\frac {d-\alpha} {\alpha}$ we prove the bound of the theorem.
\end{proof}

As a consequence of Theorem \ref{thm:sym1}, Rogers' bound for the covering density implies that $g_{\alpha}(B)\leq \frac {d^d} {\alpha^\alpha (d-\alpha)^{d-\alpha}} d\ln d(1+o(1))$ for all $\alpha>\frac d 2$ (here $o(1)$ is a function of $d$ and does not depend on $\alpha$).

For a standard covering number $g_0$, it is not important how large the size of each translate is so we may as well replace each homothet by the interior of $B$. When $\alpha>0$ this is not anymore the case. In fact, it appears that if we use the strategy of Rogers from \cite{rog57} but keep the small size of some of the translates and adjust the covering parameters then we can improve the bound on $g_{\alpha}$, even if $\alpha$ is not larger than $\frac d 2$.

\begin{theorem}
For all natural $d$ and all $\alpha$ such that $0\leq \alpha \leq d-1$, $$g_\alpha(B)\leq 2^d (d-\alpha)(\ln d + \ln\ln d + 2 + o(1))$$ for any centrally symmetric $d$-dimensional body $B$ ($o(1)$ is a function of $d$ and does not depend on $\alpha$).
\end{theorem}

\begin{proof}
Without loss of generality we assume the volume of $B$ is 1 and its center is 0. We will construct a periodic covering of the space $\mathbb{R}^d$ by translates of $B$, i.e. there will be a full-dimensional lattice $\Lambda$ such that for any translate $B'$ of $B$ and any vector $x\in\Lambda$, $x+B'$ also belongs to the covering. We consider only such lattices $\Lambda$ that 1) for any translate $B'$ of $B$ and any two vectors $x,y\in\Lambda$, $x+B'$ and $y+B'$ do not have common points and 2) for any translates $B'$ and $B''$ of $B$ and any two vectors $x,y\in\Lambda$, $x+B'$ and $y+B'$ cannot both have common points with $B''$. Both conditions can be granted by having the minimal vector of $\Lambda$ longer than 2 diameters of $B$. These two conditions essentially allow us to solve the problem on the torus $T=\mathbb{R}^d/\Lambda$. The preimage of a covering of the torus defines a periodic covering of the space and a proper covering of any translate of $B$. With a small abuse of notation, by $B$ we mean both the initial convex body and its image under the standard quotient mapping to $T$.

Our construction will depend on parameters $\lambda\in(0,1)$, $\delta\in(0,1)$, $M>0$, $N\in\mathbb{N}$. These parameters will be defined later. We consider a lattice $\Lambda$ with determinant $M$ satisfying the conditions above. $M$ can be taken arbitrarily large so such $\Lambda$ definitely exists. Let $x_1$, $\ldots$, $x_N$ be independent random variables uniformly distributed over the torus $T$. Using these random vectors, we construct $N$ random translates of $\lambda B$: $x_1+\lambda B$, $\ldots$, $x_N+\lambda B$. For each individual translate, the probability of each point of $T$ to not be covered by this translate is precisely $1-\frac {\lambda^d} M$. Since the variables are independent, the expected value of the volume of the torus not covered by any of the random translates is $M\left(1-\frac {\lambda^d} M\right)^N$. Hence we can choose concrete $N$ translates of $\lambda B$ in $T$ such that the total volume of $T$ not covered by them is not greater than $M\left(1-\frac {\lambda^d} M\right)^N\leq M e^{-\frac {\lambda^d N} M}$.

In the empty part of the torus, we construct a saturated packing by translates of $\delta B$, i.e. interiors of these translates do not intersect neither each other, nor interiors of already chosen translates of $\lambda B$ and it is impossible to add one more translate of $\delta B$ to the set. Since it is a packing, the volume calculation implies that the number of the translates of $\delta B$ is not greater than $\frac M {\delta^d} e^{-\frac {\lambda^d N} M}$. On the other hand, since the packing is saturated, when we substitute all translates of $\lambda B$ by translates of $(\lambda+\delta) B$ and all translates of $\delta B$ with translates of $2\delta B$ with the same centers, they form a covering of $T$. Indeed, otherwise there is an uncovered point $x\in T$ and $x+\delta B$ can be added to the packing so there is a contradiction to the saturation condition.

Overall, we constructed a covering of $T$ by $N$ translates of $(\lambda+\delta) B$ and no more than $\frac M {\delta^d} e^{-\frac {\lambda^d N} M}$ translates of $2\delta B$. For the next step we take a random translate $x+B$, where $x$ is uniformly distributed over $T$, and calculate the expected value of the sum of $\alpha$-powers of the homothety coefficients of those translates intersecting $x+B$. We note that $x+B$ intersects $y+\beta B$ for any fixed $y\in T$ and $\beta\in(0,1)$ if and only if $x\in y + (1+\beta) B$. Thus the probability of this happening is not greater than $\frac {(1+\beta)^d} M$. The expected value of the sum of $\alpha$-powers in our case is then no greater than

$$\frac {(1+\lambda+\delta)^d} {M} N (\lambda+\delta)^{\alpha} + \frac {(1+2\delta)^d} {M} \frac M {\delta^d} e^{-\frac {\lambda^d N} M} (2\delta)^{\alpha} =$$

$$= \frac {(1+\lambda+\delta)^d} {M} N (\lambda+\delta)^{\alpha} + \frac {(1+2\delta)^d 2^{\alpha}} {\delta^{d-\alpha}} e^{-\frac {\lambda^d N} M}.$$

Since we want an upper bound for $g_{\alpha}(B)$, all homothety coefficients should be less than 1, i.e. $\lambda+\delta<1$ and $2\delta<1$. Then there exists a concrete translate of $B$ so that the sum of $\alpha$-powers of homothets covering it satisfies this inequality. What is left is to choose suitable parameters $N, M, \lambda, \delta$. We denote $\lambda^d \frac N M$ by $\theta$. For any $\theta>0$ and any $\lambda\in(0,1)$, we can choose large enough $M$ and integer $N$ satisfying $\theta=\lambda^d \frac N M$ so that the required lattice properties hold. The inequality holds for all $\lambda<1-\delta$ so, taking the limit $\lambda\rightarrow 1-\delta$, we get that the inequality is valid for $\lambda=1-\delta$ as well. Taking all this into account, we get

$$g_{\alpha}(B)\leq \frac {2^d} {(1-\delta)^d} \theta +  \frac {(1+2\delta)^d 2^{\alpha}} {\delta^{d-\alpha}} e^{-\theta}.$$

Taking $\delta=\frac 1 {d\ln d}$ and $\theta = -(d-\alpha) \ln\delta = (d-\alpha) (\ln d + \ln \ln d)$, we conclude the proof:

$$g_{\alpha}(B) \leq \frac {2^d} {\left(1-\frac 1 {d\ln d}\right)^d} (d-\alpha)  (\ln d + \ln \ln d) + \left(1+\frac {2} {d\ln d}\right)^d 2^{\alpha} = $$

$$\leq 2^d (d-\alpha) (\ln d + \ln \ln d) \left(1 + \frac 1 {\ln d} + o\left(\frac 1 {\ln d}\right)\right) + 2^d (d-\alpha)(1 + o(1)) = $$

$$=2^d (d-\alpha) (\ln d + \ln \ln d + 2 + o(1)).$$
\end{proof}

\section{Acknowledgments}
The author would like to thank the anonymous referee for suggesting Corollary \ref{cor:ball} and the general help on improving the text. The author was supported in part by NSF grant DMS-1400876. This material is partially based upon work supported by the National Science Foundation under Grant DMS-1439786 while the author was in residence at the Institute for Computational and Experimental Research in Mathematics in Providence, RI, during the Spring 2018 semester.

\bibliographystyle{amsalpha}

\end{document}